\documentclass{amsart}
\usepackage{amsmath}
\usepackage{amssymb}
\usepackage{color}
\usepackage{mathdots}

\usepackage{graphicx}
\usepackage{epsfig}
\usepackage{multirow}

\definecolor{lightgray}{gray}{0.9}

\DeclareMathOperator{\im}{Im}

\DeclareMathOperator{\E}{E}

\DeclareMathOperator{\Char}{char}

\DeclareMathOperator{\diag}{diag}

\DeclareMathOperator{\tr}{tr}
\DeclareMathOperator{\rk}{rk}

\DeclareMathOperator{\N}{\mathbb{N}}
\DeclareMathOperator{\Z}{\mathbb{Z}}
\DeclareMathOperator{\F}{\mathbb{F}}

\newtheorem{theorem}{Theorem}[section]
\newtheorem{lemma}[theorem]{Lemma}
\newtheorem{proposition}[theorem]{Proposition}
\newtheorem{corollary}[theorem]{Corollary}

\theoremstyle{definition}
\newtheorem{definition}[theorem]{Definition}
\newtheorem{example}[theorem]{Example}
\newtheorem{problem}[theorem]{Problem}

\theoremstyle{remark}
\newtheorem{remark}[theorem]{Remark}
\numberwithin{equation}{section}

\def\separa{\hbox to 14 truecm{\hrulefill}}

\date\today

\author[A. N. Abyzov]{Adel N. Abyzov}
\address{Kazan (Volga Region) Federal University, Kazan, Russia}
\email{Adel.Abyzov@kpfu.ru}

\author[A. R. Chekhlov]{Andrey R. Chekhlov}
\address{Faculty of Mathematics and Mechanics, Tomsk State University, Tomsk, Russia}
\email{cheklov@math.tsu.ru}

\author[P. V. Danchev]{Peter V. Danchev}
\address{Institute of Mathematics and Informatics, Bulgarian Academy of Sciences, Sofia, Bulgaria}
\email{danchev@math.bas.bg}

\author[D. T. Tapkin]{Danil T. Tapkin}
\address{Kazan (Volga Region) Federal University, Kazan, Russia}
\email{danil.tapkin@yandex.ru}

\begin{document}

\title[Rings with Clean-Like Properties ...]{Rings with Clean-Like Properties: \\ Endomorphism, Matrix and Structural Theorems}

\maketitle

\begin{abstract} We investigate three clean-like properties for arbitrary rings, for endomorphism rings of abelian groups and for matrix rings over finite fields. Specifically, we study and establish when a ring is weakly strongly $k$-nil-clean for some fixed natural number $k\geq 2$, when the matrix ring is either quasi $2$-nil-clean or quasi $3$-nil-clean, and when the endomorphism ring is weakly clean.

Our theorems improve substantially on some results due to Goldsmith-V\'amos in Rend. Sem. Mat. Univ. Padova (2007), Breaz {\it et al}. in Linear Algebra \& Appl. (2013), Ko\c{s}an-Zhou in Front. Math. China (2016), Su {\it et al}. in J. Algebra \& Appl. (2027), and some other existing results in this current topic.
\end{abstract}


\bigskip




\section{Introduction and Principal Facts}

Throughout the text, all rings are associative and unital. As usual, for any such a ring $R$, the symbol $U(R)$ denotes the {\it unit group} of $R$, the symbol $Id(R)$ denotes the set of idempotent elements of $R$, the symbol $nil(R)$ denotes the set of nilpotent elements of $R$, and the symbol $C(R)$ denotes the {\it center} of $R$. Traditionally, the symbol $M_n(R)$ stands for the full {\it matrix ring} over $R$ of natural size $n\geq 2$.

Our main motivating tools in writing up the present article are the following four papers, namely \cite{BCDM13}, \cite{GV}, \cite{Su} and \cite{quasi-clean}. In fact, certain brief historical facts show that in \cite{GV} the authors examine when endomorphism rings of torsion Abelian groups are clean in the sense that every element is a sum of a unit and an idempotent (for some other results of this type in which are concerned mixed group and torsion-free groups, we also suggest sources \cite{C} and \cite{S}). Later on, in \cite{BCDM13}, the authors initiated the study of when a matrix ring over a field is nil-clean in the sense that every element is a sum of a nilpotent and an idempotent, and then apply their main result to study when the endomorphism ring of a finite rank Abelian group is nil-clean (notice that the problem for the case of infinite rank is still open) as well as of an arbitrary Abelian group is strongly nil-clean in the sense that the existing nilpotent and idempotent elements commute with one another (for a publication of this kind, we also propose source \cite{ABD}).

In the other vein, in \cite{quasi-clean} the authors introduced the concept of a {\it quasi-idempotent} in a ring $R$ as being an element $t\in R$ such that $t^2=tv$ with $v\in U(R)\cap C(R)$, and after that the notions of a {\it (strongly) quasi-clean} ring as those rings whose elements are a sum of a unit and a quasi-idempotent (that commute with one another). Some interesting results about the matrix ring are obtained there. Very recently, motivated by this definition and the establishments from \cite{BCDM13}, in \cite{Su} the authors looked at the "nil case" of quasi-cleanness by studying the so-called {\it nil-quasi-clean} matrix rings over fields refining the chief result of \cite{BCDM13} in a new light.

Very recently, in \cite{DM26} was defined the notion of {\it weakly strongly 2-nil-clean} rings as a common expansion of {\it strongly nil-clean} rings, where the latter were completely characterized, respectively, in \cite{BDZ} and \cite{KZ16}, as in the first source only the abelian case was described. Concretely, a weakly strongly $2$-nil-clean ring is the one for which each element is the sum of the difference of two idempotents plus a nilpotent, and it was established in \cite{DM26} that a ring $R$ is weakly strongly $2$-nil-clean if, and only if, $R\cong R_1\times R_2$, where $R_1$ is strongly $2$-nil-clean, and either $R_2=\{0\}$ or $J(R_2)$ is nil and $R_2/J(R_2)\cong \F_5$.

\medskip

Thus, inspired by all of the presented above, we are managing here to explore in-depth when an arbitrary ring is weakly strongly $k$-nil-clean for any $k\geq 2$, as well as when endomorphism and matrix rings are either weakly clean and quasi nil-clean by trying to characterize the corresponding rings of these three directions comprehensively and significantly. In this way, our work is organized into three basic sections, and the major results are arranged as follows: In the second section, we try to discover the structure of weakly clean abelian groups adapting some ring-theoretic techniques (see Theorem~\ref{redcomp} and Proposition~\ref{2-groups}). In the third section, we attack in two subsections how the property quasi nil-cleanness is structured in the case of matrix rings over a former finite field (see Theorems~\ref{lambda-2-N} and \ref{lambda-3-N}, respectively). In the fourth section, we try to look at the more general situation of several proper generalizations of strong nil-cleanness (see Theorem~\ref{tha2}). In the final fifth section, consulting with the established thus far results, we discuss some relevant material that could unify and stimulate a further extensive exploration on the considered three themes.

\section{Weakly Clean Endomorphism Rings of Abelian Groups}

Imitating the terminology from \cite{GV}, we shall say that an abelian group $G$ is {\it weakly clean} if its corresponding endomorphism ring $\E(G)$ is weakly clean in the sense that each element is a sum or a difference of a unit and an idempotent (compare with \cite{AA} and \cite{D} for the corresponding purely ring-theoretic definition of weak cleanness as the firstly cited article is devoted to the commutative case only), that is, each endomorphism of $G$ is a sum or a difference of an invertible endomorphism (= automorphism) of $G$ and an idempotent endomorphism (= projection) of $G$.

\medskip

The following technicality extends the corresponding \cite[Lemma 1]{GV}.

\begin{lemma}\label{direct}
If $A$ is a weakly clean group and $B$ is a clean group, then $A\oplus B$ is a weakly clean group.
\end{lemma}

\begin{proof} It is well known that the endomorphisms of the direct sum $A\oplus B$ may be regarded as matrices of the form $m=\begin{pmatrix} x & y \\ z & w\end{pmatrix}$, where $x\in\mathrm{End}(A)$, $w\in\mathrm{End}(B)$, $y\in\mathrm{Hom}(B, A)$ and $z\in\mathrm{Hom}(A, B)$ (cf. \cite{F1}).

Now, as $A$ is weakly clean, we can write $x = u\pm c$, where $u$ is a unit and $c$ is an idempotent. So, one checks that $w-zu^{-1}y\in\mathrm{End}(B)$. But, since $B$ is clean, we can also write $w-zu^{-1}y=v\pm d$ with $v$ a unit,
$d$ an idempotent, and the sign on $\pm d$ is aligned with the sign on $\pm c$ (this can be done, because in a clean ring every of its elements can simultaneously be written as $u+e$ and as $u'-e'$, where $u,u'$ are units and $e,e'$ are idempotents).

Now, one writes that
$$m=\begin{pmatrix} u & y \\ z & v+zu^{-1}y\end{pmatrix}\pm\begin{pmatrix} c & 0 \\ 0 & d\end{pmatrix},$$
where the second term is obviously an idempotent. We note that the first term is a unit: in fact, pre-multiplying it by $p=\begin{pmatrix} 1 & 0 \\ -zu^{-1} & 1\end{pmatrix}$ and post-multiplying it by
$q=\begin{pmatrix} u^{-1} & -u^{-1}yv^{-1} \\ 0 & v^{-1}\end{pmatrix}$ leads to the identity matrix and, therefore, both $p$, $q$ are invertible, as required.
\end{proof}

However, the direct sum of two weakly clean groups need {\it not} be weakly clean too. In this aspect, the next consequence is worthy of documentation.

\begin{corollary}\label{fullinv}
If $G=A\oplus B$, where $B$ is fully invariant in $G$ and is clean, then $G$ is weakly clean if, and only if, $A$ is weakly clean.
\end{corollary}

\begin{proof} It easily follows that $E(A)$ and $E(B)$ both are endomorphic images of $E(G)$, giving the assertion.
\end{proof}

Since all divisible groups are always clean in virtue of \cite[Lemma 3]{GV}, then we obtain from Lemma~\ref{direct} the following.

\begin{corollary}
A group is weakly clean if, and only if, its reduced part is weakly clean.
\end{corollary}

That is why, we henceforth will consider only {\bf reduced} weakly clean groups.

\medskip

Let $G$ be a $p$-group. Recall that the two-sided ideal
$$H(G)=\{\varphi\in E(G)\, |\, {\rm height}(x)<\infty\ \Rightarrow\ {\rm height}(\varphi(x)) > {\rm height}(x)\ \forall \ x\in G[p]\}$$
is often referred to as the \emph{Pierce radical} of $G$. For a reduced $p$-group $G$, it is always true that $J(E(G))\leq H(G)$ (cf. \cite{P63}, \cite{KMT}), where $J(R)$ stands for the Jacobson radical of a ring $R$.

\medskip

Now, we set
\[
J_r'(R):=\{a\in R\,|\, \mathrm{for\, each}\ r\in R\
\mathrm{the\, element}\ 1-ar \ \mathrm{or}\ 1+ar
\ \mathrm{is\, a\, unit\, of\, ring}\ R\}.
\]
and, similarly,
\[
J_l'(R):=\{a\in R\,|\, \mathrm{for\, each}\ r\in R\
\mathrm{the\, element}\ 1-ra \ \mathrm{or}\ 1+ra
\ \mathrm{is\, a\, unit\, of\, ring}\ R\}.
\]

It is quite clear that $J_r'(R)R\subseteq J_r'(R)$, $RJ_l'(R)\subseteq J_l'(R)$ and $J(R)\subseteq J'(R):=J_r'(R)\cap J'_l(R)$, as well as that $f(J'(R))\subseteq J'(S)$ for any homomorphism $f: R\to S$ of rings $R, S$.

\medskip

The following technicalities are critical keys for our further presentation.

\begin{lemma}
If $G$ is a reduced weakly clean $p$-group, then $H(G)\subseteq J'(E(G))$.
\end{lemma}

\begin{proof} Choose $\varphi\in H(G)$ and write $\varphi=u \pm \pi$, where $\pi^2=\pi$. Same as \cite[Lemma 5]{GV}, we must have $\pi =1$, so that $1\pm\varphi\in U(E(G))$, and since $H(G)$ is ideal, it follows that $\varphi\in J'(E(G))$, as requested.
\end{proof}

\begin{lemma}\label{threepoints}
(1) Let $G=A\oplus B$. If $\alpha'\in J'(E(A))$, then $\alpha\in J'(E(G))$, where $\alpha\!\upharpoonright\!A=\alpha'$ and $\alpha(B)=\{0\}$. If $\alpha\in J'(E(G))$, then $\pi\alpha i\in J'(E(A))$, where $\pi: G\to A$ is a projection and $i:A\to G$ is a natural embedding.

(2) If $G=\bigoplus_{n\geq 1}\langle a_n\rangle$, where ${\rm order}(a_n)=p^{r_n}$, $r_1<r_2<\dots$, then
$H(G)\nsubseteq J'(E(G))$.

(3) If a reduced $p$-group $G$ is weakly clean, then $G$ does not have a direct summand which is an unbounded direct sum of cyclic groups.
\end{lemma}

\begin{proof} (1) We know that the ring $E(G)$ is isomorphic to the ring $R$ of all matrices of size 2, say
$(\varphi_{ij})$, where $\varphi_{11}\in E(A)$. We have to prove that $(\lambda_{ij})\in J'(R)$, where
$\lambda_{11}\in J'(E(A))$ and $\lambda_{12}=\lambda_{21}=\lambda_{22}=0$. To that end, suppose $(g_{ij})\in R$
and $1\pm \lambda_{11}g_{11}=u_{11}\in U(E(A))$; then, $1\pm (\lambda_{ij})(g_{ij})=(v_{ij})$, where
$v_{11}=u_{11}$, $v_{12}=\pm\lambda_{11}g_{12}$, $v_{21}=0$, $v_{22}=1$ noticing that $(v_{ij})\in U(R)$.

Similarly, it can be shown that $1\pm (g_{ij})(\lambda_{ij})\in U(R)$ for every $(g_{ij})\in R$. So, what remains to establish is that, if $(\lambda_{ij})\in J'(R)$, then $\lambda_{11}\in J'(E(A))$. Choose $g_{11}\in E(A)$ and take
$(g_{ij})\in R$ such that $g_{12}=g_{21}=g_{22}=0$. But, by assumption, $1\pm (\lambda_{ij})(g_{ij})\in U(R)$ and
$1\pm (g_{ij})(\lambda_{ij})\in U(R)$ and, consequently, $1\pm \lambda_{11}g_{11}\in U(E(A))$ and $1\pm g_{11}\lambda_{11}\in U(E(A))$.

(2) Choosing $\alpha\in E(G)$ such that $\alpha(a_n)=p^{r_{n+1}-r_n}a_{n+1}$, it is rather evident that $\alpha\in H(G)$. But, $1\pm \alpha\notin U(E(G))$ since $a_1\notin\im(1\pm\alpha)$, and thus $\alpha\notin J'(E(G))$.

(3) It follows immediately from (1) and (2).
\end{proof}

In \cite[Lemma 8]{GV} is shown that all bounded $p$-groups are clean. Since each reduced unbounded countable $p$-group must have a direct summand which is a direct sum of cyclic groups whose orders are also unbounded, then we quickly derive the following consequence.

\begin{corollary}
A countable weakly clean $p$-group $G$ is bounded.
\end{corollary}

Likewise, in \cite[Theorem 10]{GV} is proven that an unbounded totally projective $p$-group must have a direct summand which is an unbounded direct sum of cyclic groups, and so we automatically extract the following more general consequence.

\begin{corollary}
A totally projective weakly clean $p$-group is bounded.
\end{corollary}

Besides, it was established in \cite[Lemma 8]{GV} that even torsion-complete $p$-groups are always clean. So, with Lemma~\ref{threepoints} at hand, we detect the following consequence.

\begin{corollary}
A direct sum of torsion-complete $p$-groups is weakly clean if, and only if, it is torsion-complete (and hence clean).
\end{corollary}

We now focus on the torsion-free case that has {\it not} been considered in \cite{GV}.

\medskip

If $G$ is a torsion-free group, then let $\Pi(G):=\{p\in\mathbb{P}\,|\,pG\neq G\}$. Such a group $G$ is said to be {\it quasi-homogeneous}, provided that $\Pi(G)=\Pi(H)$ for each non-zero pure subgroup $H\leq G$.

\medskip

We also recall that

\medskip

\[
H(G)=
\left\{\varphi\in E(G)\,
\Bigg|
\begin{array}{l}
\,x\in D \Rightarrow \varphi(x)=0,\\
 x\in G\setminus D, {\rm height}_p(x)<\infty \Rightarrow {\rm height}_p(x)<{\rm height}_p(\varphi(x))\, \mathrm{for\, each}\, p\in\mathbb{P}
\end{array}
\right\}.
\]

\medskip

Thus, if $G$ is a torsion-free group of finite rank, then $H(G)\leq J(E(G))$ (see \cite[Lemma 21.1]{KMT} and \cite{K}).

\medskip

We, thereby, come to the following.

\begin{lemma}
If $G$ is a torsion-free weakly clean group, then either $1+f\in U(R)$ or $1-f\in U(R)$ for every $f\in H(G)$.
\end{lemma}

\begin{proof} Take $f\in H(G)$ and write $f=u \pm e$, where $u$ is a unit and $e$ is an idempotent. For $x\in G$
and $p\in\Pi(G)$ with ${\rm height}_p(x)<\infty$, as ${\rm height}_p(x)<{\rm height}_p(f(x))$, we get ${\rm height}_p((u-f)(x))={\rm height}_p(x)$. So, the difference $u-f$ preserves the $p$-height of each element in
$G$ for all primes $p$. We also claim that $\ker(u-f)=\{0\}$. Indeed, if $0\neq x\in D$, then $(u-f)(x)=u(x)\neq 0$,
and if $x\not\in D$, then ${\rm height}_p(x)<\infty$ for some $p$. Thus, ${\rm height}_p((u-f)(x))={\rm height}_p(x)$ and so $(u-f)(x)\neq 0$. Since $u-f=\mp e$, it follows that $\ker(e)=\{0\}$, and since $e$ is an idempotent, we infer that $e=1$, as wanted.
\end{proof}

We proceed by recalling the following.

\begin{lemma} \cite[Lemma 2.1]{LZ}
Let $R$ be a commutative indecomposable ring. Then:

(1) $R$ is clean if, and only if, $R$ is local.

(2) $R$ is weakly clean but not clean if, and only if, $R$ has exactly two maximal ideals and $2\in U(R)$.

(3) If $2\not\in U(R)$, $R$ is weakly clean if, and only if, $R$ is clean.
\end{lemma}

The next comments are worthwhile.

\begin{remark}\label{comments}
From previous statement, it follows that:

(1) If $G$ is a reduced torsion-free clean group of rank 1, then $E(G)$ is local, i.e., $E(G)\cong\mathbb{Z}_p$
for some prime $p$, where $\mathbb{Z}_p$ is the ring of all rational numbers with denominators that are mutually prime with $p$.

(2) If $G$ is a torsion-free weakly clean non-clean group of rank 1, then $E(G)\cong\mathbb{Z}_{p,q}$ for some odd primes $p$, $q$, where $\mathbb{Z}_{p,q}$ is the ring of all rational numbers with denominators that are mutually prime with $p$, $q$, and $2G=G$.

(3) (see \cite{S}) A reduced completely decomposable torsion-free group $G$ is clean if, and only if, $G=\bigoplus_p G_p$, where $G_p$ are of finite rank the homogeneous components of $G$, and $qG_p=G_p$ for any prime $q\neq p$.
\end{remark}

Note that if for a ring $R$ we have $a\not\in J'(R)$ $\Rightarrow$ $a\in U(R)$, then $R$ is weakly clean.

\begin{lemma}
Let $R$ be a weakly clean ring. Then:

(1) Each right ideal $I\nsubseteq J_r'(R)$ contains a non-zero idempotent.

(2) Each left ideal $I\nsubseteq J_l'(R)$ contains a non-zero idempotent.

(3) If the idempotent $e$ is primitive, the corner ring $eRe$ is weakly clean.
\end{lemma}

\begin{proof}
(1) Let $a\in R$, and suppose that $aR$ contains only the zero idempotent. For $r\in R$, if $ar = u + e$, where
$u$ is a unit in $R$ and $e$ is an idempotent in $R$, then
\[
u(1-e)u^{-1} = (ar-e)(1-e)u^{-1} = ar(1-e)u^{-1}\in aR.
\]
Thus, $e=1$ and $1-ar=-u$. From equality $ar=u-e$, we analogously conclude $1+ar=u$. So, $a\in J_r'(R)$.

(2) Same as in (1), it must be that $ra=u\pm 1$, and thus $a\in J_l'(R)$.

(3) Let $f$ be a primitive idempotent, and choose $a\in fRf$ with $a\not\in J'(fRf)$. Thanks to (1), there exists
$0\neq g = g^2\in aR\leq fR$, and since $f$ is chosen primitive, we find that $gR=fR$. Hence, $aR = fR$, so that $f = ab$ for some $b\in R$ and $a$ has right inverse $fbf$ in $fRf$. Similarly, $a$ possesses left inverse, whence $a\in U(fRf)$.

Now, taking into account the fact that the elements from $J'_r(R)$ or $J'_l(R)$ are weakly clean, it follows by what we have proved so far that $fRf$ is weakly clean ring, as desired.
\end{proof}

As a direct consequence, we deduce the following.

\begin{corollary}\label{indecomp}
An indecomposable direct summand of a weakly clean group too is weakly clean.
\end{corollary}

The following statement is a bit curious.

\begin{lemma}
Let $G=\bigoplus_{n\geq 1}G_n$, where $G_n$ are reduced torsion-free groups of rank 1 with types $t(G_n)\leq t(G_{n+1})$, and $pG_n\neq G_n$ for some prime $p$. Then, $G$ is not a weakly clean group.
\end{lemma}

\begin{proof} Choose $0\neq g_n\in G_n\setminus pG_n$ with characteristics $\chi(g_n)\leq\chi(g_{n+1})$. Define a mapping $f: G\to G$ via $f(g_n)=pg_{n}+p^2g_{n+1}$. By a way of contradiction, we write that $f=u\pm e$, where
$u$ is a unit in $\E(G)$ and $e$ is an idempotent in $\E(G)$. If $e\neq 1$, then one inspects that $e(x)=0$ for some
$x\in G\setminus pG$. Therefore, $px=u(x)$; however, $u(x)\in G\setminus pG$ since $u$ is an automorphism. So, $e=1$.
Thus, $f\mp 1=u$. But, it is readily verified that $g_1\notin\im(f\mp 1)$, a contradiction, as suspected.
\end{proof}

The following construction confirms our observation alluded to above that the direct sum of two weakly clean groups is {\it not} necessarily again weakly clean.

\begin{example}\label{twonon}
If $A$, $B$ are two non-isomorphic weakly clean but not clean torsion-free groups of rank 1, then $A\oplus B$ is {\it not} weakly clean.
\end{example}

\begin{proof}
Utilizing Remark~\ref{comments}, one writes that $R:=E(A)\cong\mathbb{Z}_{p,q}$ and $S:=E(B)\cong\mathbb{Z}_{g,t}$
for some odd primes $p,q,g,t$, where, for instance, $g\neq p,q$ as $A\ncong B$.

Note that, if $\alpha=(2p-q)/(q-p)$, then $\alpha\in U(R)$; but, $\alpha+1=p/(q-p)$, $\alpha+2=q/(q-p)$ ensuring $\alpha+1,\alpha+2\not\in U(R)$.

Moreover, if $\beta=(2g-t)/(g-t)$, then $\beta\in U(S)$; but, $\beta-1=g/(g-t)$, $\beta-2=t/(g-t)$ assuring $\beta-1,\beta-2\not\in U(S)$.

Consequently, $(\alpha+1)\oplus(\beta-1)$ is the asked endomorphism of $A\oplus B$ that is {\it not} weakly clean. In fact, if $(\alpha+1)\oplus(\beta-1)=\delta\pm e$, then we have $\delta=\varphi\oplus\psi$ for some $\varphi\in U(R)$,
$\psi\in U(S)$ and an idempotent $e=e_1\oplus e_2$. Since $$(\alpha+1)=\varphi\pm e_1\not\in U(R), (\beta-1)=\psi\pm e_2\not\in U(S),$$ we get $e_1, e_2\neq 0$, so that $e_1=1$ and $e_2=1$. But, $\alpha+2\not\in U(R)$ and $\beta-2\not\in U(S)$, thus seeing that the equality $(\alpha+1)\oplus(\beta-1)=\delta\pm (1,1)$ is imposable, which really means that the direct sum $A\oplus B$ is {\it not} weakly clean, as pursued.
\end{proof}

Let $G$ be a reduced completely decomposable torsion-free group such that each its direct summand of rank 1 is weakly clean (see Corollary~\ref{indecomp}). Then, $G$ can be written as $G=G_1\oplus G_2$, where $G_1$ is fully invariant in $G$. Indeed, assume that $G_1$ is the direct sum of all clean direct summands of rank 1, and $G_2$ is the direct sum of all weakly clean (but non-clean) direct summands of rank 1. It follows from Remark~\ref{comments} that $G_1$ is fully invariant in $G$. Also, each homogeneous component of $G_2$ is fully invariant in $G_2$.

\medskip

The following statement is crucial for characterizing when a weakly clean group is reduced completely decomposable torsion-free.

\begin{proposition}\label{rankone}
Let $G=A^n$ for some reduced torsion-free weakly clean non-clean group $A$ of rank 1 and some $n\geq 1$. If $G$ is weakly clean, then $n=1$.
\end{proposition}

\begin{proof} Since by assumption $A$ is reduced non-clean, one has that $pA\neq A$ and $qA\neq A$ for some different odd prime $p,q$. Write $pm-qn=1$ for some integers $m,n$, where $n>1$, and suppose $\alpha$ is the multiplication of $A$ on $pm$, $\beta$ is the multiplication of $A^{n-1}$ on $-pm$ by setting $f:=\alpha\oplus\beta$.

Assume now that $f=u\pm e$, where $u$ is a unit and $e$ is an idempotent. If $e(x)=0$ for some $x\in G\setminus pG$, then $f(x)=u(x)$ which is an absurd, because $${\rm height}_p(f(x))>{\rm height}_p(u(x))={\rm height}_p(x),$$ so that $\ker e=\{0\}$ implying $e=1$ as $e$ is invertible in view of finite rank restriction on $G$.

Write now that $f=u+1$. Since $f(A)\leq A$, one sees that $u(A)\leq A$, i.e., $u\!\upharpoonright\!A$ is an  automorphism of $A$. Thus, $$pma-a=(pm-1)a=qna=u(a)$$ for some $0\neq a\in A$ that is wrong, because $${\rm height}_q(qna)>{\rm height}_q(u(a))={\rm height}_q(a).$$

The other equality $f=u-1$ is also impossible, because $$-pmx+x=(-pm+1)x=-qnx$$ for some $x\in A^{n-1}$ yielding $${\rm height}_q((-pm+1)x)={\rm height}_q(qnx)>{\rm height}_q(u(x))={\rm height}_q(x).$$ Therefore, it must be that $n=1$, as expected.
\end{proof}

We now accumulated all the information necessary to establish the following criterion.

\begin{theorem}\label{redcomp}
Let $G$ be a reduced completely decomposable torsion-free group. Then, $G$ is weakly clean if, and only if, $G=G_1\oplus G_2$, where $G_1$ is a completely decomposable clean group and $G_2$ is a weakly clean group of
rank $\leq 1$.
\end{theorem}

\begin{proof} {\bf Necessity.} Suppose that $G_1$ is the direct sum of all clean direct summands of rank $1$ of $G$, and $G_2$ is the direct sum of all weakly clean but non-clean direct summands of rank $1$ of $G$. It follows from Remark~\ref{comments} that $G_1$ is fully invariant in $G$ and, in particular, $G_1$ is clean, as claimed, whereas $G_2$ is weakly clean (in case that it is non-zero).

Furthermore, one observes that each homogeneous component of $G_2$ is fully invariant in $G_2$, so it must be weakly clean and hence Proposition~\ref{rankone} works to deduce that this component is of rank $\leq 1$. Moreover, each direct summand of $G_2$ is weakly clean as being a fully invariant direct summand of a weakly clean group, whence Example~\ref{twonon} guarantees that ${\rm rank}(G_2)=1$ provided that $G_2\neq\{0\}$, as asserted.

{\bf Sufficiency.} It follows at one from Corollary~\ref{fullinv}.
\end{proof}

Commenting on the last necessary and sufficient condition, a difficult question that remains unsettled yet is to decide whether or not homogeneous completely decomposable weakly clean groups are of finite rank.

\medskip

As we have seen above, there is a torsion-free weakly clean group which is {\it not} clean. Nevertheless, the following claim concerning the 2-primary case somewhat sounds surprisingly. Before formulating and proving it, we need the following statement.

\begin{lemma}\label{newinstrum} (\cite[Lemma 20.3]{KMT})
Let $G$ be a $p$-group, and let $\alpha\in \E(G)$. In order to have $\alpha\in U(\E(G))$, it is necessary and sufficient $\ker\alpha\cap G[p]=\{0\}$ and $\alpha((p^nG)[p])=(p^nG)[p]$ for each $n=0,1,2,\dots$.
\end{lemma}

So, we are ready to check truthfulness of the following surprising assertion.

\begin{proposition}\label{2-groups}
Weakly clean $2$-groups are clean.
\end{proposition}

\begin{proof} Assume that $\alpha=u-e$, where $u$ is an automorphism and $e$ is an idempotent. Then, one may write that $\alpha=(u-2e)+e$, where $u-2e\in U(\E(G))$. To substantiate the stated above relation, one observes that, if $x\in G[2]$, then $(u-2e)(x)=u(x)\neq 0$, i.e., $\ker(u-2e)=\{0\}$, and thus $$(u-2e)((2^nG)[2])=u((2^nG)[2])=(2^nG)[2].$$ It now just remains to refer to Lemma~\ref{newinstrum} to get the promised implication.
\end{proof}

So, a difficult question which logically arises is what happens for weakly clean $p$-groups in the case of $p\not=2$; we conjecture that there is a weakly clean $3$-group that is {\it not} clean, but the formal building of such a construction still eludes us, however. Perhaps the endomorphism ring of such a group, if certainly it exists, should {\it not} be additively generated by automorphisms of this group.

\section{Variations of Quasi Nil-Clean Matrix Rings over a Finite Field}

Standardly, an element $r$ of a ring $R$ is said to be {\it $m$-potent} if there is an integer $m>1$ such that $r^m=r$. In particular, if $m=2$, the element $r$ is termed {\it idempotent}, and if $m=3$, the element $r$ is termed {\it tripotent}.

More generally, mimicking \cite{quasi-clean}, an element $e$ of a ring $R$ is called {\it quasi idempotent} if $e^{2} = \lambda e$ for some ${\lambda \in U(R) \cap C(R)}$. By analogy, an element $f$ of a ring $R$ is called {\it quasi $q$-potent} if $f^{q} = \lambda f$ for some ${\lambda \in U(R) \cap C(R)}$. If $e^{2} = \lambda e$ with ${\lambda \in U(R) \cap C(R)}$, then one simply checks that $(\lambda^{-1} e)^{2} = \lambda^{-1}e$. Therefore, the set of quasi idempotents coincides with the set of elements $kf$ with ${k \in U(R) \cap C(R)}$ and $f^{2} = f$. In particular, every element of a field is necessarily a quasi idempotent. Moreover, if $f$ is a $q$-potent and ${\lambda \in U(R) \cap C(R)}$, then one plainly verifies that $(\lambda f)^{q} = \lambda^{q} f = \lambda^{q-1} (\lambda f)$ and $\lambda f$ is a quasi $q$-potent. However, for $q > 2$, not every quasi $q$-potent can be represented in this form.

\medskip

The following illustrates some more elementary properties of this sort.

\begin{example}
(1) For any $q > 1$, every quasi $q$-potent of an integral domain $R$ (and, therefore, every element of a direct product of integral domains) has the form $\lambda f$ for some $\lambda \in U(R)$ and a $q$-potent $f$. Indeed, assume that $g \in R$ and $g^{q} = \mu g$ for some $\mu \in U(R)$. Thus, if $g = 0$, then $g = 1 \cdot 0$; otherwise, $g$ is a unit and $g = g \cdot 1$.

(2) Consider the matrix
$A :=
\begin{pmatrix}
0 & -1
\\
1 &  0
\end{pmatrix}
\in M_{2}(\mathbb{F}_{3})$. Since one calculates that $A^{3} = -A$, so $A$ is a quasi $3$-potent. It is also clear that, if $(\lambda A)^3 = \lambda A$ for some $\lambda \in U(\mathbb{F}_{3})$, then $\lambda^{3} = -\lambda$, a contradiction. Therefore, $A$ cannot be represented as $\lambda f$, where ${\lambda \in U(M_{2}(\mathbb{F}_{3})) \cap C(M_{2}(\mathbb{F}_{3}))}$ and $f \in M_{2}(\mathbb{F}_{3})$ is a tripotent.
\end{example}

We now come to the following notion, which was introduced and explored in \cite{quasi-clean} as a non-trivial generalization of clean rings (that are those rings whose elements are a sum of a unit and an idempotent).

\begin{definition}
An element $a$ of a ring $R$ is called {\it (strongly) quasi clean} if $a$ is a sum of a unit and a quasi idempotent (that commute one another). So, a ring is called {\it (strongly) quasi clean} if each of its elements is (strongly) quasi clean.
\end{definition}

For example, it was shown there that if the matrix ring $M_{n}(R)$ over a commutative ring $R$ is quasi clean, then $R$ is quasi clean. Likewise, it was proven that for a prime $p$ and a positive integer $n \geq 2$ the ring $M_{n}(\mathbb{Z}_{(p)})$ is strongly quasi clean if, and only if, $p > n$.
\medskip

Also, we recall that a ring $R$ is called {\it (strongly) $q$-nil-clean} if each of its elements is a sum of a $q$-potent and a nilpotent (that commute). We now introduce two generalizations of this concept.

\begin{enumerate}
\item A ring $R$ is called {\it (strongly) $\lambda$-$q$-nil-clean} if each of its elements is the sum $\lambda f + N$, where $f^{q}=f$, ${\lambda \in U(R) \cap C(R)}$, and $N$ is a nilpotent (that commute).
\item A ring $R$ is called {\it (strongly) quasi $q$-nil-clean} if each of its elements is the sum of a quasi $q$-potent and a nilpotent (that commute).
\end{enumerate}

The following chain of implications is pretty obvious:
\[
  \text{$R$ is $q$-nil-clean} \;\Rightarrow\;
  \text{$R$ is $\lambda$-$q$-nil-clean} \;\Rightarrow\;
  \text{$R$ is quasi $q$-nil-clean}.
\]

These three classes are pairwise distinct as our next arguments show: every infinite field is $\lambda$-$q$-nil-clean, but {\it not} $q$-nil-clean. Since $\mathbb{Z}/6\mathbb{Z}$ is isomorphic to a direct product of fields, every element of $\mathbb{Z}/6\mathbb{Z}$ is quasi idempotent, and hence $\mathbb{Z}/6\mathbb{Z}$ is quasi $2$-nil-clean. However, $\mathbb{Z}/6\mathbb{Z}$ is manifestly {\it not} $\lambda$-$2$-nil-clean, because there are no nonzero nilpotent elements in $\mathbb{Z}/6\mathbb{Z}$ and $2 + 6\mathbb{Z}$ cannot be represented as $\lambda f$, where $\lambda \in U(R)$ and $f$ is an idempotent.

Our next commentaries are somewhat worthy of recording.

\begin{remark}
Quasi $2$-nil-clean rings were studied in \cite{Su} under the name {\it nil-quasi-clean rings}. It was proved in \cite[Theorem 2.3]{Su} that, for any field $F$, the matrix ring $M_{n}(F)$ is quasi-$2$-nil-clean if, and only if, either $n=2$ and $F$ is a perfect field of characteristics $2$, or $n > 2$ and $F \cong \mathbb{F}_{2}$. For the sake of completion we will include a much simpler direct proof of this result in the sequel.

In another vein, it is trivial that, for any field $F$, the ring $M_{n}(F)$ is $\lambda$-$q$-nil-clean if, and only if, each element of $M_{n}(F)$ is a linear combination over $C(M_{n}(F))$ of a $q$-potent and a nilpotent. However, in the general case, this equivalence does {\it not} hold: in fact, $\mathbb{Z}/6\mathbb{Z}$ is {\it not} $\lambda$-$2$-nil-clean, but each element of $\mathbb{Z}/6\mathbb{Z}$ is a linear combination over $\mathbb{Z}/6\mathbb{Z}$ of an idempotent and a nilpotent (say, $z = z \cdot 1 + 0$).
\end{remark}

Let $F$ be an arbitrary field. We use the following notations:

\begin{itemize}
\item $I_{k}$ is the {\it unit $k \times k$ matrix}.
\item $\overline{F}$ is the {\it algebraic closure} of $F$.
\item $V_{\lambda}(A)$ is the {\it eigenspace} of a matrix $A \in M_{n}(F)$ associated with the {\it eigenvalue} $\lambda$ of $A$ amounting to ${V_{\lambda} (A) = \ker(A-\lambda I_{n})}$.
\item Let $p = x^{n} - \sum\limits_{i=0}^{n-1}a_{i}x^{i} \in F[x]$ be a unitary polynomial. Denote the {\it companion matrix} of $p$ by $C(p)$; i.e.,
\[
  C(p) =
  \left(
  \begin{array}{ccccc}
  0 & 0 & \cdots & 0 & a_{0}
  \\
  1 & 0 & \cdots & 0 & a_{1}
  \\
  0 & 1 & \cdots & 0 & a_{2}
  \\
  \vdots & \vdots & \ddots & \vdots & \vdots
  \\
  0 & 0 & \cdots & 1 & a_{n-1}
  \end{array}
  \right).
\]

\medskip

\noindent Define the {\it trace} of $p$ as the {\it trace} of $C(p)$; i.e., $\tr(p) = \tr(C(p)) = a_{n-1}$.
\end{itemize}

\subsection{Matrix rings of order $2$}

We will start our research with the special case of matrix rings of order $2$ by proving the following.

\begin{proposition}\label{lambda-odd-N-M2}
Let $q \geq 3$ be an odd integer. Then, for any field $F$, the ring $M_{2}(F)$ is $\lambda$-$q$-nil-clean.
\end{proposition}

\begin{proof}
Take $A \in M_{2}(F)$. Without loss of generality, we may assume that $A$ is in a rational canonical form. There are three different cases:

\medskip

(a) $A = \lambda I_{2}$ for some $\lambda \in F$. Thus, $A = \lambda \cdot I_{2} + 0$ is a decomposition as a linear combination of a $q$-potent and a nilpotent.

\medskip

(b) $A = C(p)$ with $p = x^{2} - ax - b \in F[x]$ and $a \not = 0$. So,
\[
  A =
  \left(
  \begin{array}{cc}
  0 & b
  \\
  1 & a
  \end{array}
  \right)
  =
  a \cdot
  \left(
  \begin{array}{cc}
  0 & ba^{-1}
  \\
  0 & 1
  \end{array}
  \right)
  +
  \left(
  \begin{array}{cc}
  0 & 0
  \\
  1 & 0
  \end{array}
  \right)
\]
is a decomposition as a linear combination of a $q$-potent and a nilpotent.

\medskip

(c) $A = C(p)$ with $p = x^{2} - b \in F[x]$. Considering decomposition
\[
  A =
  \left(
  \begin{array}{cc}
  0 & b
  \\
  1 & 0
  \end{array}
  \right)
  =
  \left(
  \begin{array}{cc}
  -1 & b+1
  \\
  0 & 1
  \end{array}
  \right)
  +
  \left(
  \begin{array}{cc}
  1 & -1
  \\
  1 & -1
  \end{array}
  \right),
\]
one verifies that the last term is obviously nilpotent, whereas
$
\begin{pmatrix}
  -1 & b+1
  \\
  0 & 1
\end{pmatrix}
$
is a $q$-potent in any field $F$, as needed.
\end{proof}

Contrary to Proposition \ref{lambda-odd-N-M2}, the case of even $q$ presents strong restriction on a field $F$.

\begin{proposition}\label{lambda-even-N-M2}
Let $q \geq 2$ be an even integer. For a field $F$, the following four statements are equivalent:
\begin{enumerate}
\item $M_{2}(F)$ is quasi $q$-nil-clean;
\item $M_{2}(F)$ is $\lambda$-$q$-nil-clean;
\item $F$ is a perfect field of characteristic $2$.
\end{enumerate}
\end{proposition}

\begin{proof}
It is quite obvious that (2) guarantees (1), so we skip details.

\medskip

$(1) \Rightarrow (3)$. Assume that $\Char F \not= 2$ and $M_{2}(F)$ is quasi $q$-nil-clean. Put
\[
  A =
  \left(
  \begin{array}{cc}
    1 & 0
    \\
    0 & -1
  \end{array}
  \right) \in M_{2}(F).
\]
By assumption, there exist $f \in M_{2}(F)$, $\lambda \in F^{*}$ and a nilpotent $N \in M_{n}(F)$ such that
$A = f + N$ and $f^{q} = \lambda f$. We thus have
\[
  0 = \tr A = \tr (f + N) = \tr f.
\]
Since $\Char F \not= 2$, we discover that either $f = 0$, or a matrix $f$ has two distinct nonzero eigenvalues $\pm \mu \in \overline{F}$. In the former case, $A$ is a nilpotent, a contradiction. In the latter case, we must have $$\lambda = (-\mu)^{q-1} = (-1)^{q-1} \cdot (\mu)^{q-1} = -\lambda,$$ a contradiction.

\medskip

Assume now that $\Char F = 2$, but $F$ is not a perfect field. Fix $b \in F$ that is not a square.
Set
\[
  A  =
  \left(
  \begin{array}{cc}
    0 & b
    \\
    1 & 0
  \end{array}
  \right) \in M_{2}(F).
\]
If there exist $f \in M_{2}(F)$, $\lambda \in F^{*}$ and a nilpotent $N \in M_{n}(F)$ such that
$A = f + N$ and $f^{q}=\lambda f$, then
\[
  0 = \tr A = \tr(f + N) = \tr f.
\]
Since $A$ is not a nilpotent, we extract $f \not=0$. But $\Char F = 2$ and evenness of $q$ imply diagonalizability of $f$, because $f$ is similar to $\mu I_{2}$ in $M_{n}(\overline{F})$ for some $\mu \in \overline{F}$. Hence, $f = \mu I_{2}$ and
\[
  N = A - f =
  \left(
  \begin{array}{cc}
    -\mu & b
    \\
    1 & -\mu
  \end{array}
  \right).
\]
However, $\det N = 0$ ensures $b = \lambda^{2}$, a contradiction.

\medskip

$(3) \Rightarrow (1)$. Let $A \in M_{2}(F)$. Without loss of generality, we may assume that $A$ is in a rational canonical form. Consider four basic cases:

\medskip

(a) $A = \lambda I_{2}$ for some $\lambda \in F$. Then, $A = \lambda \cdot I_{2} + 0$ is a decomposition as a linear combination of a $q$-potent and a nilpotent.

\medskip

(b) $A = C(p)$ with $p = x^{2} - ax - b \in F[x]$ and $a \not = 0$. Thus,
\[
  A =
  \left(
  \begin{array}{cc}
  0 & b
  \\
  1 & a
  \end{array}
  \right)
  =
  a \cdot
  \left(
  \begin{array}{cc}
  0 & ba^{-1}
  \\
  0 & 1
  \end{array}
  \right)
  +
  \left(
  \begin{array}{cc}
  0 & 0
  \\
  1 & 0
  \end{array}
  \right)
\]
is a decomposition as a linear combination of a $q$-potent and a nilpotent.

\medskip

(c) $A = C(p)$ with $p = x^{2} - b \in F[x]$ and $b \not = 0$. Since $F$ is a perfect field, there exists $z \in F$ such that $z^{2} = b$. So,
\[
  A =
  \left(
  \begin{array}{cc}
    0 & b
    \\
    1 & 0
  \end{array}
  \right)
  =
  z \cdot I_{2}
  +
  \left(
  \begin{array}{cc}
    z & b
    \\
    1 & z
  \end{array}
  \right)
\]
is a decomposition as a linear combination of a $q$-potent and a nilpotent.

\medskip

(d) $A = C(x^2)$. This matrix is apparently a nilpotent itself, as required.
\end{proof}

\subsection{Quasi $2$-nil-clean rings}

As it was mentioned earlier, a ring $R$ is quasi $2$-nil-clean if, and only if, it $R$ is $\lambda$-$2$-nil-clean.

Let $F$ be a field. Exploiting Proposition \ref{lambda-even-N-M2}, the ring $M_{2}(F)$ is quasi $2$-nil-clean if, and only if,$F$ is a perfect field of characteristic $2$. For matrix rings of order $n > 2$ we arrive at following much stricter criterion.

\begin{theorem}\label{lambda-2-N}
For a field $F$ and an integer $n > 2$, the following three points are equivalent:
\begin{enumerate}
\item $M_{n}(F)$ is quasi $2$-nil-clean;
\item $M_{n}(F)$ is nil-clean;
\item $F \cong \mathbb{F}_{2}$.
\end{enumerate}
\end{theorem}

\begin{proof} The equivalence $(2) \Leftrightarrow (3)$ was proved in \cite{BCDM13}. Since every nil-clean ring is quasi $2$-nil-clean, it is enough to establish only that (1) forces (3).

$(1) \Rightarrow (2)$. Assume that $M_{n}(F)$ is quasi $2$-nil-clean. Choose an arbitrary element $b \in F$ and take the matrix
\[
  A :=
  \left(
  \begin{array}{c|c}
    I_{n-1} & 0
    \\
    \hline
    0 & b
  \end{array}
  \right) \in M_{n}(F).
\]
By hypothesis, there exist an idempotent $e \in M_{n}(F)$, $\lambda \in F$ and a nilpotent $N \in M_{n}(F)$ such that
$A = \lambda e + N$. Since $\dim V_{1}(A) = n-1 > n/2$, the matrices $A$ and $e$ have common eigenvector $v$: $Av = v$ and $ev = kv$ with $k \in \{0,1\}$. Hence, $v$ is a eigenvector for $N$. Computing $Av$, we obtain
\[
  1 = \lambda k.
\]
Consequently, $k = \lambda = 1$. Thus,
\[
  A = e + N
\]
is a decomposition as a sum of an idempotent and a nilpotent. It thereby follows that $\tr (A) = \tr e$ lies in a prime subfield of $F$. However, $b$ can be any element of a field $F$. Therefore, $F$ is a prime field, as desired.

\medskip

Assume that $\Char F \not = 2$ and put $b = -1$. By computing traces, we deduce
\[
  (n-2) \cdot 1_{F} = \rk e \cdot 1_{F},
\]
i.e., $\rk e \equiv n-2 \mod \Char F$. Since $\Char F \not = 2$, we infer $\rk e \leq n-2$. It thus follows that $$\dim V_{1} (A) + \dim V_{0}(e) \geq (n-1) + 2$$ and, therefore, $A$ and $e$ have common eigenvector $v$ such that $Av = v$ and $ev = 0$. So, $Nv = (A-e)v = v$, a contradiction. Consequently, $F$ is a prime field of characteristic $2$, that is, $F \cong \mathbb{F}_{2}$, as wanted.
\end{proof}

\subsection{Quasi $3$-nil-clean rings}

We present here some common generalizations to some of the results from \cite{Su}. Our chief result here is the following criterion.

\begin{theorem}\label{lambda-3-N}
For a field $F$ and an integer $n \geq 3$, the following four issues are equivalent:
\begin{enumerate}
\item $M_{n}(F)$ is quasi $3$-nil-clean;
\item $M_{n}(F)$ is $\lambda$-$3$-nil-clean;
\item $M_{n}(F)$ is $3$-nil-clean;
\item $F \cong \mathbb{F}_{2}$ or $F \cong \mathbb{F}_{3}$.
\end{enumerate}
\end{theorem}

\begin{proof} The equivalence $(3) \Leftrightarrow (4)$ was established in \cite{A17}. Since every $3$-nil-clean ring is $\lambda$-$3$-nil-clean, and $\lambda$-$3$-nil-clean ring is quasi $3$-nil-clean, it is enough to prove that (1) gives (4).

$(1) \Rightarrow (4)$. Assume that $M_{n}(F)$ is quasi $3$-nil-clean. Take an arbitrary element $b \in F$ and choose the matrix
\[
  A :=
  \left(
  \begin{array}{c|c}
    I_{n-1} & 0
    \\
    \hline
    0 & b
  \end{array}
  \right) \in M_{n}(F).
\]
By hypothesis, there exist $f \in M_{n}(F)$, $\lambda \in F^{*}$ and a nilpotent $N \in M_{n}(F)$ such that
$A = f + N$ and $f^{3} = \lambda f$.

\medskip

\noindent{\bf Claim 1.} The field $F$ is prime.

\medskip

Assume that matrices $A$ and $f$ have common eigenvector $v \in \overline{F}^{n}$ such that $Av = v$. Hence, $v$ is an eigenvector for $N$ and $fv = (A-N)v = v$. From $f^{3} = \lambda f$, it immediately follows that $\lambda=1$ and $f$ is a tripotent. Therefore, $b$ is in a prime subfield of $F$ as $$b = \tr(A) - (n-1) \cdot 1_{F} =\tr(f) - (n-1) \cdot 1_{F}.$$

Therefore it is enough to prove that $A$ and $f$ have common eigenvector $v \in \overline{F}^{n}$ such that $Av = v$. Consider two possible cases:

\medskip

(a) $\Char F \not = 2$.

\medskip

In this case, $f$ is diagonalizable over $\overline{F}$ and has at most $3$ pairwise distinct eigenvalues in $\overline{F}$, namely $0$, $\mu$ and $-\mu$, where $\mu^{2}=\lambda$. Since $\dim V_{1}(A) = n-1$, for all $n>3$ matrices $A$ and $f$ have common eigenvector $v \in \overline{F}^{n}$ with $Av = v$. Thus, all elements $b \in F$ are in a prime subfield of $F$, i.e., $F$ is a prime field, as asserted.

\medskip

Assume now that $n=3$. In this case, either of $A$ and $f$ have common eigenvector $v \in \overline{F}^{n}$ with $Av = v$, or $f$ is similar in $M_{n}(\overline{F})$ to $\diag(-\mu, 0, \mu)$. In the latter case, one finds that $\tr (A) = \tr(f) = 0$. In the former case, $\tr (A)$ is in a prime subfield of $F$. Therefore, $F$ is a prime field too, as claimed.

\medskip

(b) $\Char F = 2$. In this case, $f$ has at most $2$ pairwise distinct eigenvalues in $\overline{F}$, namely $0$ and $\mu$, where $\mu^{2}=\lambda$. Also, one knows that the Jordan normal form of $f$ is triangular with blocks of sizes $1$ and $2$.

\medskip

If $n>5$, matrices $A$ and $f$ necessary have common eigenvector $v \in \overline{F}^{n}$ with $Av = v$.

\medskip

If $n=4$, matrices $A$ and $f$ necessary have common eigenvector $v \in \overline{F}^{n}$ with $Av = v$, unless $f$ is similar in $M_{n}(\overline{F})$ to the matrix
\[
  \left(
  \begin{array}{cc|cc}
  0 & 1 & 0 & 0
  \\
  0 & 0 & 0 & 0
  \\
  \hline
  0 & 0 & \mu & 1
  \\
  0 & 0 & 0 & \mu
  \end{array}
  \right).
\]
In the latter case, $\tr(A) = \tr(f) = 0$ and $b$ is in a prime subfield of $F$. Thus, we conclude that $F$ is a prime field.

\medskip

If $n=3$, matrices $A$ and $f$ either have common eigenvector $v \in \overline{F}^{n}$ with $Av = v$, or $f$ is similar in $M_{n}(\overline{F})$ to one of the following matrices
\[
  \left(
  \begin{array}{cc|c}
  \mu & 1 & 0
  \\
  0 & \mu & 0
  \\
  \hline
  0 & 0 & 0
  \end{array}
  \right),
  \quad
  \left(
  \begin{array}{cc|c}
  0 & 1 & 0
  \\
  0 & 0 & 0
  \\
  \hline
  0 & 0 & \mu
  \end{array}
  \right).
\]
If $f$ is similar to the first matrix, then $\tr(A) = \tr(f) = 0$ and $b$ is in a prime subfield of $F$. If $f$ is similar to the second matrix, then $b = \tr(A) = \mu$. So, all elements of the field $F$, except perhaps one, lie in a prime subfield. Consequently, the field $F$ is prime indeed.

\medskip

\noindent{\bf Claim 2.} The field $F$ is isomorphic either to $\mathbb{F}_{2}$ or $\mathbb{F}_{3}$.

\medskip

Since $F$ is a prime field, it is enough to show only that $\Char F \not \in \{2,3\}$. To that end, assume the contrary. Consider two major cases:

\medskip

(a) $n > 3$. As we showed earlier, $f$ is a tripotent. Since $\dim V_{1}(A) = n-1$, both $\dim V_{-1} (f)$ and $\dim V_{0} (f)$ are at most one, as for otherwise matrices $A$ and $f$ would have common eigenvector $v$ with $Av = v$ and $fv \not = v$. Therefore, the tripotent $f$ is similar to the one of the following four matrices
\[
  \left(
  \begin{array}{c|c|c}
    I_{n-2} & 0 & 0
    \\
    \hline
    0 & -1 & 0
    \\
    \hline
    0 & 0 & 0
  \end{array}
  \right),
  \quad
  \left(
  \begin{array}{c|c}
    I_{n-1} & 0
    \\
    \hline
    0 & -1
  \end{array}
  \right),
  \quad
  \left(
  \begin{array}{c|c}
    I_{n-1} & 0
    \\
    \hline
    0 & 0
  \end{array}
  \right),
  \quad
  I_{n}.
\]
Thus, $\tr f \in \{(n-3) \cdot 1_{F}, (n-2) \cdot 1_{F}, (n-1) \cdot 1_{F}, n \cdot 1_{F}\}$.

Next, put $b = -3 \cdot 1_{F}$. We now derive that $\tr f = (n-4) \cdot 1_{F}$ and $n-4 \equiv n-k \mod \Char F$, where $k \in \{0,1,2,3\}$. It follows at once that $\Char F \mid 4-k$, a contradiction.

\medskip

(b) $n=3$. As it was proved earlier, if $\tr(A) \not = 0$, then $A = f + N$ is a decomposition as a sum of a tripotent and a nilpotent. If $\tr(A) = 0$, then $A$ is a tripotent itself and we can put $f = A$, $N = 0$. Similarly to the case (a), we detect $\Char F \mid 4-k$ for $b \not = -3 \cdot 1_{F}$, a contradiction which concludes the entire argumentation.
\end{proof}

It was shown earlier that, for $q \in \{2,3\}$ and a field $F$, the ring $M_{n}(F)$ is $\lambda$-$q$-nil-clean if, and only if, $M_{n}(F)$ is quasi $q$-nil-clean. That is why, the following question is sensible.

\begin{problem}
Let $F$ be a field and let $q \geq 4$ be an integer. Is it true that, for every $n\geq 2$, the ring $M_{n}(F)$ is $\lambda$-$q$-nil-clean if, and only if, $M_{n}(F)$ is quasi $q$-nil-clean?
\end{problem}

\subsection{Some remarks on the general case}

It is quite hard to say anything specific about general case. In fact, take integers $q > 3$ and large enough $n$. Assume, moreover, that $M_{n}(F)$ is quasi $q$-nil-clean and $\overline{F}$ is an algebraic closure of $F$.

Besides, take arbitrary $b \in F$ and put
\[
  A =
  \left(
  \begin{array}{c|c}
    I_{n-1} & 0
    \\
    \hline
    0 & b
  \end{array}
  \right) \in M_{n}(F).
\]
By hypothesis, there exist $f \in M_{n}(F)$, $\lambda \in F^{*}$ and a nilpotent $N \in M_{n}(F)$ such that
$A = f + N$ and $f^{q} = \lambda f$. Note that $f$ has at most $q$ pairwise distinct eigenvalues in $\overline{F}$ and $\dim V_{1}(A) = n-1$. Furthermore, consider two possible cases as follows:

\medskip

(a) $\Char F  = 0$.

\medskip

In this case, quasi $q$-potent $f$ is diagonalizable if $M_{n}(\overline{F})$. Thus, for $n > q$, matrices $A$ and $f$ have common eigenvector $v \in \overline{F}^{n}$ with $Av = v$. It therefore follows that $f$ is a $q$-potent and
\[
  A = f + N
\]
is a decomposition as a sum of a $q$-potent and a nilpotent. However, trace of a $q$-potent matrix $f \in M_{n}(\overline{F})$ can can only take a finite number of values. Hence, $\tr A = \tr f$ tells us that $|F| < \infty$, a contradiction.

\medskip

(b) $\Char F = p>0$.

\medskip

The matrix $f$ has Jordan normal form over $\overline{F}$. Take an arbitrary eigenvalue $\mu \in \overline{F}$ of $f$ and consider one of the corresponding Jordan blocks $$J_{m}(\mu) = \mu I_{m} + J_{m}(0) \in M_{m}(\overline{F}).$$ Put $p^{k}$ as the greatest nonnegative power of $p$ such that $p^{k} \mid q-1$. Since $J_{m}(\mu)^{q} = \lambda J_{m}(\mu)$, we have $1 \leq m \leq p^{k}$. Likewise, since there are $1+\frac{q-1}{p^k}$ different $q$-potents in $\overline{F}$, for any $n > (1+\frac{q-1}{p^k}) \cdot p^k = (q-1) + p^k$, the matrices $A$ and $f$ have common eigenvector $v \in \overline{F}^{n}$ with $Av = v$. It thus follows that $f$ is a $q$-potent and
\[
  A = f + N
\]
is a decomposition as a sum of a $q$-potent and a nilpotent. In particular, $\tr A = \tr f$. If, for some eigenvalue $1 \not = \mu \in \overline{F}$ of $f$, we have $\dim V_{\mu}(f) > 1$, then $1-\mu$ is an eigenvalue of $N$, a contradiction. Consequently, $\dim V_{\mu}(f) \leq 1$.

If now $p \nmid q-1$, then every Jordan block of $f \in M_{n}(\overline{F})$ is scalar and so $\tr f$ cannot take more than $2^{q-1}$ different values. If, however, $p \mid q-1$, then there are $1+\frac{q-1}{p^k}$ different $q$-potents in $\overline{F}$. For an eigenvalue $0,1 \not = \mu \in \overline{F}$ it is pretty clear that $\tr J_{m}(\mu) \in \mu\mathbb{F}_{p}$. Since there is no Jordan block $J_{m}(0)$ for $m>1$ in any $q$-potent, $\tr f$ cannot take more than $2 \cdot p^{-1 + (q-1)/p^k}$ different values. At the same time, $\tr f$ as a sum of $q$-potents lies in a splitting field of $x^q-x$ over $\mathbb{F}_{p}$.

\medskip

Thus, summarizing the conclusions presented above, we have shown the following general claim.

\begin{proposition}
\label{general-case-quasi-q-nil}
Let $F$ be a field, and let $q > 3$ be an integer.
\begin{enumerate}
\item
If for some $n > q$ the ring $M_{n}(F)$ is quasi $q$-nil-clean, then $\Char F$ is finite.

\item Let $\Char F = p < \infty$ and $p^{k}$ is the greatest non-negative power of $p$ such that $p^{k} \mid q-1$. If, for some $n > (q-1) + p^{k}$, the ring $M_{n}(F)$ is quasi $q$-nil-clean, then $F$ is a finite field and $F$ is a subfield of splitting field of $x^q-x$ over $\mathbb{F}_{p}$. Moreover,
\begin{enumerate}
\item if $\mathrm{gcd}(p, q-1) = 1$, then $|F| \leq 2^{q-1}$;
\item if $p \mid q-1$, then $|F| \leq 2 \cdot p^{-1 + (q-1)/p^k}$.
\end{enumerate}
\end{enumerate}
\end{proposition}

In other words, for any $q > 3$, there exists a finite set $A_{q}$ consisting of finite fields such that, for any large enough $n$, the quasi $q$-nil-cleanness of $M_{n}(F)$ implies $F \in A_{q}$ up to an isomorphism. However, it is still unclear whether every element of $F$ has to be a $q$-potent or not. Nevertheless, Proposition~\ref{general-case-quasi-q-nil} allows us to answer this problem for $q=4$.

\begin{corollary}
\label{quasi-4}
For a field $F$ and an integer $n \geq 7$, the following four issues are equivalent:
\begin{enumerate}
\item $M_{n}(F)$ is quasi $4$-nil-clean;
\item $M_{n}(F)$ is $\lambda$-$4$-nil-clean;
\item $M_{n}(F)$ is $4$-nil-clean;
\item $F \cong \mathbb{F}_{2}$ or $F \cong \mathbb{F}_{4}$.
\end{enumerate}
\end{corollary}

\begin{proof}
$(1) \Rightarrow (4)$. Let $p = \Char F$ and $q = 4$. Since $n > (q-1) + 3$, we can apply Proposition~\ref{general-case-quasi-q-nil}~(2) to get the assertion.

\medskip

(a) If $\mathrm{gcd}(p, 3)$ = 1, then $|F| \leq 2^3 = 8$. Thus, $p \in \{2,5,7\}$. Let $f \in M_{n}(F)$ be a $q$-potent as in case (b) before Proposition \ref{general-case-quasi-q-nil}; thereby, for any eigenvalue $1 \not= \mu$ of $f$, we know that $\dim V_{\mu}(f) \leq 1$.

\smallskip

If $p = 2$, then a field $F$ is a subfield of a splitting field of $x^{4}-x$ over $\mathbb{F}_{2}$, i.e., $F \subseteq \mathbb{F}_{4}$.

\smallskip

If $p = 5$, then $F = \mathbb{F}_{5}$. So, the polynomial $x^4 - x \in \mathbb{F}_{5}[x]$ decomposes as $x(x-1)(x^2 + x + 1)$. Set $\alpha, \beta \in \overline{F}$ to be the roots of $x^2 + x + 1$. Since the polynomial $x^2 + x + 1 \in \mathbb{F}_{5}[x]$ is irreducible, $\tr f$ can only take one of the following values:

\begin{gather*}
((n-3) + (0 + \alpha + \beta)) \cdot 1_{\mathbb{F}_{5}},
  \\
((n-2) + (\alpha + \beta)) \cdot 1_{\mathbb{F}_{5}},
  \\
((n-1) + (0)) \cdot 1_{\mathbb{F}_{5}},
  \\
n \cdot 1_{\mathbb{F}_{5}},
\end{gather*}
Therefore, $\tr f$ can take only $3 < |\mathbb{F}_{5}|$ different values -- a contradiction.

\smallskip

If $p = 7$, then $F = \mathbb{F}_{7}$. There are already four $4$-potents in $\mathbb{F}_{7}$, namely $0, 1, 2, 4$. Hence, $\tr f$ can only take one of the following values:

\begin{gather*}
((n-3) + (0 + 2 + 4)) \cdot 1_{\mathbb{F}_{7}}, 
    \\
((n-2) + (0 + 2)) \cdot 1_{\mathbb{F}_{7}},\quad 
((n-2) + (0 + 4)) \cdot 1_{\mathbb{F}_{7}},\quad 
((n-2) + (2 + 4)) \cdot 1_{\mathbb{F}_{7}}, 
    \\
((n-1) + (0)) \cdot 1_{\mathbb{F}_{7}},\quad 
((n-1) + (2)) \cdot 1_{\mathbb{F}_{7}},\quad 
((n-1) + (4)) \cdot 1_{\mathbb{F}_{7}}, 
    \\
n \cdot 1_{\mathbb{F}_{7}}.  
\end{gather*}
Consequently, $\tr f$ can take only $6 < |\mathbb{F}_{7}|$ different values -- a contradiction.

\smallskip

(b) If $p \mid q-1$, then $p=3$ and $|F| \leq 2 \cdot 3^{-1 + 3/3} = 2$. Thus, there is obviously no such a field $F$.

\medskip

$(4) \Rightarrow (3)$ is precisely \cite[Theorem 2]{A17}. Also, implications $(3) \Rightarrow (2)$ and $(2) \Rightarrow (1)$ are quite trivial, so their verification is voluntarily dropped off.
\end{proof}

Similarly to the case of $\mathbb{F}_{5}$ in the Corollary~\ref{quasi-4}, one can specifically restate Proposition~\ref{general-case-quasi-q-nil}~(2b) as follows: if $\mathrm{gcd}(\Char F,q-1) = 1$ and $n > q$, then $|F| < 2^{t-1}$, where $t$ is the number or irreducible factors of $x^{q}-x \in F[x]$.

\section{\bf Weakly Strongly k-Nil-Clean Rings}

Let $k>1$ be a natural number. Generalizing the basic notion introduced in \cite{DM26}, we say that a ring $R$ is {\it weakly strongly $k$-nil-clean} if, for every element $r\in R$, there are $e_1,\ldots, e_k\in Id(R)$ and $n \in nil(R)$ that all commute such that $r = \pm e_1\pm \ldots \pm e_k + n$.

\medskip

The following technical claim is a necessary starting point of view.

\begin{lemma}\label{la1}
Let $k>1$ be a natural number and let $R$ be a weakly strongly $k$-nil-clean ring. If $p={\rm char}(R)$ is a prime number, then there exists a natural number $t$ such that $r^{p^{t}}=r^{p^{t+1}}$ for any $r \in R$.
\end{lemma}

\begin{proof} Choose $r\in R$. According to the stated conditions, the equality $$r= \varepsilon_1+ \ldots +\varepsilon_k + j$$ holds, where $\varepsilon_i\in \{\pm e_i\}$ for every $1\leq i\leq k$, $e_1,\ldots, e_k\in Id(R)$, $j \in nil(R)$, and the elements $e_1,\ldots, e_k, j$ commute with each other. Moreover, for some natural number $t$, the equalities $j^{p^{t}}=j^{p^{t+1}}=0$ also hold. Since $\varepsilon_i^{p^{t}}=\varepsilon_i^{p^{t+1}}$, one can write that $$r^{p^{t}}=e_1^{p^{t}}+\ldots+e_k^{p^{t}}=e_1^{p^{t+1}}+\ldots+e_k^{p^{t+1}}=r^{p^{t+1}},$$ as needed.
\end{proof}

As it is well-known, a ring $R$ is called {\it strongly $\pi$-regular} if, for every element $r\in R$, there is an element $s\in R$ such that $r^n=sr^{n+1}$ for some $n\in\N$.

\medskip

Our main achievement is this section is the following structural result.

\begin{theorem}\label{tha2} Let $k>1$ be a fixed natural number. If $R$ is a weakly strongly $k$-nil-clean ring, then $R$ is a strongly $\pi$-regular ring which is isomorphic to a finite direct product of rings $S$, each of which satisfies the following conditions:
\begin{enumerate}
\item[(a)] $J(S)$ is a nil-ideal;
\item[(b)] Every right primitive ideal $I$ of $S$ is maximal, and the quotient ring $S/I\cong\F_p$, where $p\leq 2k+1$;
\item[(c)] For some prime $p$, the quotient ring $S/J(S)$ is a subdirect product of the field $\F_p$ and $p\leq 2k+1$;
\item[(d)] If for some natural number $m$ the ring $(\F_p)^m$ is not a weakly strongly $k$-nil-clean ring, then $S/J(S)\cong (\F_p)^k$, where $k<m$.
\end{enumerate}
\end{theorem}

\begin{proof}
We intend to show that there will exist smallest natural number $N$ for which the equality $N\cdot 1_R=0$ is true. To that purpose, for some commuting idempotents $e_1, \ldots, e_{k}$ and a nilpotent $j$ from the ring $R$, the equality $$(k+1)\cdot 1_R=\varepsilon_1+\ldots+\varepsilon_k+j$$ is fulfilled, where $\varepsilon_i\in \{\pm e_i\}$ for each $1\leq i\leq k$. Furthermore, for an arbitrary subset $A$ of the set $I=\{1,\ldots, k\},$ we put $$e_A:=\prod_{i\in A}e_i \prod_{i\in I \setminus A}(1-e_i).$$ Since
$$
1_R=\prod_{i\in I}(e_i+(1-e_i))=\sum_{A\in 2^I}e_A,
$$
one inspects that
$$
(k+1)\cdot 1_R=(\varepsilon_1+\ldots+\varepsilon_k+j)\sum_{A\in 2^I}e_A=\sum_{A\in 2^I}N_A e_A+j,
$$
where $N_A\in \Z$ and $\mid N_A\mid\leq k$ for every $A\in 2^I.$ Then,
$$
j=\sum_{A\in 2^I}(k+1-N_A ) e_A.
$$
Since $(e_A)_{A\in 2^I}$ is an orthogonal system of idempotents, for some natural number $n$, the following equalities are valid:
$$
0=j^n=\sum_{A\in 2^I}(k+1-N_A )^ne_A.
$$
Therefore, $(k+1-N_A )^ne_A=0$ for any $A\in 2^I$ and so
$$[1,\ldots, 2k+1]^n\cdot 1_R=\sum_{A\in 2^I}[1,\ldots, 2k+1]^ne_A=0.$$ Thus, indeed, there exists smallest natural number $N$ such that the equality $N\cdot 1_R=0$ is realizable, as claimed.

Writing now $N=p_1^{\alpha_1}\ldots p_m^{\alpha_m}$, where $p_1,\ldots, p_m$ are pairwise distinct primes, the classical Chinese Remainder Theorem reaches the existence of a ring isomorphism $$R\cong R_1\times \ldots\times R_m,$$ where $p_i^{\alpha_i}\cdot 1_{R_i}=0$ for every $1\leq i\leq m$. Next, for an arbitrary $1\leq i\leq m$, we manage to prove that each of the rings $R_i$ is strongly $\pi$-regular. To substantiate this, choose $r\in R_i$. So, Lemma \ref{la1} teaches us that, for some natural numbers $n_1, n_2$ with $n_1>n_2$, the equality $\overline{r}^{n_1}=\overline{r}^{n_2}$ holds, where $\overline{r}:=r+p_iR_i\in R_i/p_iR_i$. That is why, the equality $(r^{n_1}-r^{n_2})^{\alpha_i}=0$ holds as well and, consequently, for some polynomial $f(x)\in \Z[x]$, we deduce $$r^{n_2 \alpha_i}=r^{n_2 \alpha_i+1}f(r).$$ Now, since a finite direct product of strongly $\pi$-regular rings remains a strongly $\pi$-regular ring, it must be that $R$ is a strongly $\pi$-regular ring too. Also, since the Jacobson radical of an arbitrary $\pi$-regular ring is known to be a nil-ideal, $J(R)$ is too a nil-ideal.

Furthermore, we fix an arbitrary index $i\in \{1,\ldots, m\}$. We show that every right primitive factor of the ring $R_i$ is isomorphic to the field $\F_{p_i}$. In fact, let $I$ be a right primitive ideal of the ring $R_{i}$, and suppose $a$ is an arbitrary element of the ring $R_i/I$. Then, for some commuting idempotents $f_1, \ldots, f_{k}\in R_i/I$, and a nilpotent $j_0\in R_i/I$, the equality $$a=f_1'+\ldots+f_k'+j_0,$$ where $f_i'\in \{\pm f_i\}$, will hold for each $1\leq i\leq k$. But, since the ring $R_i/I$ has characteristic $p_i$, the equality $a^{p_i^n}=a^{p_i^{n+1}}$ is obviously satisfied for some natural number $n$, as required.

Suppose now that $R_i/I$ is {\it not} a division ring. Since the field $\mathbb{F}_{p^2_i}$ can be embedded in the matrix ring $M_2(\mathbb{F}_{p_i})$, there is an invertible matrix $A$ from $M_2(\mathbb{F}_{p_i})$ whose order is exactly $p_i+1$. On the other hand, from the reasoning given above and looking at \cite[11.19]{Lam01}, it follows for some $n\in \mathbb{M}$ that the equality $A^{p_i^{n+1}-p_i^n}=E$ is fulfilled, where $E$ is the standard identity matrix. If, for a moment, $p_i$ is an odd prime, then $(p_i+1, p_i^{n+1}-p_i^n)=2$ and, thereby, $A^2=E$, which contradicts the choice of matrix $A$. If, however, $p_i=2$, then $(p_i+1, p_i^{n+1}-p_i^n)=1$, which is manifestly also impossible. So, $R_i/I$ is a division ring, as suspected. Since $R_i/I$ is a weakly strongly $k$-nil-clean ring and $0,1$ are the only idempotents of $R_i/I$, one concludes that $$R_i/I\cong \F_p$$ and $p\leq 2k+1$, as formulated above.

Suppose next that, for some natural number $m$, the ring $(\F_p)^m$ is {\it not} a weakly strongly $k$-nil-clean ring. Since, same as above, every right primitive ideal is maximal, it again follows from the Chinese Remainder Theorem that the set of all right primitive ideals of the ring $R$ consists of at most $m-1$ elements. In particular, we deduce $$R_i/J(R_i)\cong (\F_p)^k,$$ where $k<m$, as asserted.
\end{proof}

A logical question of important interest, in order to obtain an eventual necessary and sufficient condition, is whether or not the statement in the last theorem can be reversed? At this stage, we are unable to conjecture the way of a probable validity.

\medskip

Further, as a series of subsequent corollaries, we extract the following ones.

\begin{corollary} If $R$ is a weakly strongly $k$-nil-clean ring for some $k\in \N$, then $J(R)=nil(R)$.
\end{corollary}

The next consequence also appeared in \cite{BDZ} in a slightly more weak form concerning only {\it abelian} rings.

\begin{corollary}\label{ca2} \cite[Theorem 1]{KZ16}
The following conditions are equivalent for a ring $R$:
\begin{enumerate}

\item[(1)] $R$ is strongly weakly nil-clean;
\item[(2)] $R$ is isomorphic to either $R_1$, $R_2$, or $R_1\times R_2$, where
\begin{enumerate}
\item[(a)] $R_1$ is strongly nil-clean, that is, $J(R_1)$ is nil and $R_1/J(R_1)$ is Boolean;
\item[(b)] $R_2/J(R_2)\cong \F_3$ and $J(R_2)$ is nil.
\end{enumerate}

\end{enumerate}
\end{corollary}

\begin{proof}
$(1) \Rightarrow (2)$. Since the element $(1, -1)\in \F_3\times \F_3$ manifestly does not have the form $\pm e$, where $e$ is an idempotent in the ring $\F_3\times \F_3$, the ring $\F_3\times \F_3$ is not strongly weakly nil-clean. The desired implication then follows directly from Theorem \ref{tha2}.

$(2) \Rightarrow (1)$. Write $R=R_1\times R_2$, where the rings $R_1$ and $R_2$ satisfy conditions (a) and (b) of point (2), respectively. We now menage to establish that $R$ is a strongly weakly nil-clean ring. To that target, assume $a\in R$ and set $\overline{a}:=a+J(R)$.

Knowing that $R/J(R)$ is isomorphic to the direct product of the field $\F_3$ and a ring which is a subdirect product of the field $\F_2$, it thus follows from the proof of Theorem \ref{tha2} that the ring $\Z[\overline{a}]$ is finite and, therefore, is contained in a subring $S$ of the ring $R/J(R)$ which is isomorphic to the ring $S=\F_2^{n}\times \F_3$, where $n\in \N$. Furthermore, a direct inspection verifies that $S$ is a strongly weakly nil-clean ring. Moreover, since $J(R)\cap \Z[\overline{a}]$ is a nilpotent ideal of the ring $\Z[\overline{a}]$, the idempotents of the ring $\Z[\overline{a}]$ lift modulo the ideal $J(R)\cap \Z[\overline{a}]$ and, thus, $a$ is the sum or the difference of a
nilpotent and an idempotent that commute with each other, as required.
\end{proof}

\begin{corollary} \cite[Theorem 2.26]{DM26}
The following conditions are equivalent for a ring $R$:
\begin{enumerate}

\item[(1)] $R$ is weakly strongly 2-nil-clean;
\item[(2)] $R$ is isomorphic to either $R_1$, $R_2$, or $R_1\times R_2$, where
\begin{enumerate}
\item[(i)] $R_1$ is strongly 2-nil-clean, that is, $J(R_1)$ is nil and $R_1/J(R_1)$ is 3-potent ring;
\item[(ii)] $R_2/J(R_2)\cong \F_5$ and $J(R_2)$ is nil.
\end{enumerate}

\end{enumerate}
\end{corollary}

\begin{proof}
$(1) \Rightarrow (2)$. Since the element $(1, 3)\in \F_5\times \F_5$ manifestly does not have the form $\pm e\pm f$, where $e, f$ are idempotents in the ring $\F_5\times \F_5$, then the ring $\F_5\times \F_5$ is not weakly strongly 2-nil-clean. Then, the implication follows immediately from Theorem \ref{tha2}.

$(2) \Rightarrow (1)$. Let the rings $R_1$ and $R_2$ satisfy conditions (i) and (ii) of point (2), respectively. It is straightforward to verify that the ring $(\F_2)^{n}\times (\F_3)^{m}\times \F_5$ is a weakly strongly 2-nil-clean ring for arbitrary $n,m\in \N$. Then, using arguments similar to those given in proving implication $(2) \Rightarrow (1)$ of Corollary \ref{ca2}, we can straightforwardly demonstrate that the direct product $R_1\times R_2$ is a weakly strongly 2-nil-clean ring, as asked for.
\end{proof}

Our final consequence is the following one.

\begin{corollary} The next conditions are equivalent for a ring $R$:
\begin{enumerate}

\item[(1)] $R$ is weakly strongly 3-nil-clean;
\item[(2)] $R$ is isomorphic to either $R_1$, $R_2$, $R_3$, or $R_1\times R_2$, or $R_1\times R_3$, where
\begin{enumerate}
\item[(i)] $R_1$ is strongly 2-nil-clean;
\item[(ii)] $R_2/J(R_2)\cong (\F_5)^m$, where $m\leq 3$ and $J(R_2)$ is nil;
\item[(iii)] $R_3/J(R_3)\cong \F_7$ and $J(R_3)$ is nil.
\end{enumerate}

\end{enumerate}
\end{corollary}

\begin{proof}
$(1) \Rightarrow (2)$. Since the element $$(1, -3)\in \F_5\times \F_7$$ manifestly does not have the form $\pm e\pm f \pm g$, where $e, f, g$ are idempotents in the ring $\F_5\times \F_7$, it follows that the ring $\F_5\times \F_7$ is not weakly strongly 3-nil-clean. It is, likewise, similarly shown that the ring $\F_7\times \F_7$ is not weakly strongly 3-nil-clean.

Besides, it is directly verified that the element $$(1, -1, 2, -2)\in (\F_5)^4$$ does not have the form $\pm e\pm f \pm g$, where $e, f, g$ are idempotents in the ring $(\F_5)^4$, whence the ring $(\F_5)^4$ is not weakly strongly 3-nil-clean. Then, the implication follows automatically from Theorem \ref{tha2}.

$(2) \Rightarrow (1)$. Let us assume that the rings $R_1$, $R_2$ and $R_3$ satisfy conditions (i), (ii) and (iii) of point (2), respectively. It is directly inspected that the rings $$(\F_2)^{n}\times (\F_3)^{m}\times (\F_5)^k$$ and $$(\F_2)^{n}\times (\F_3)^{m}\times \F_7$$ are weakly strongly 3-nil-clean for arbitrary $n,m\in \N$ and $k\leq 3$.
Thus, using arguments analogous to those given in proving implication $(2) \Rightarrow (1)$ of Corollary \ref{ca2}, we can straightforwardly illustrate that the direct products $R_1\times R_2$ and $R_1\times R_3$ are both weakly strongly 3-nil-clean rings, as pursued.
\end{proof}

\section{Concluding Discussion and Further Work}

In this final section, in regard to the posed above concrete questions and problems, we now comment some further perspectives of the results we had obtained so far, hoping eventually that their generative expansions could be true, thereby hopefully stimulating an intensive research study of the presented subject in the very near future. Explicitly, the more attractive things that we have achieved above suggest to describe those rings for which each element is the sum of a unit and a linear combination of idempotents with coefficients in $\mathbb{Z}$ commuting all together, so the leitmotif here is to ask where or not such rings could be classified in terms of the Jacobson radical and (sub)direct products of already well-characterized classes of rings. So, as a culmination of our recent considerations, we conjecture that this can occur for at least endomorphism and matrix rings happening that their complete description is achievable.

\vskip0.5pc

\noindent{\bf Funding:} The work of the second-named author, A.R. Chekhlov, was supported by the Ministry of Science and Higher Education of Russia (agreement No. 075-02-2025-1728/2) and by the Regional Scientific and Educational Mathematical Center of Tomsk State University.

\end{document}